\documentclass[12pt]{amsart}   

\usepackage{amsmath, amssymb, amsthm}   

 \usepackage{graphicx}          
 
 \usepackage{stmaryrd}  

\include{xypic}

\def\E{\mathbb{E}}
\def\P{\mathbb{P}}

\newtheorem{thm}{Theorem}
\newtheorem{prop}[thm]{Proposition}
\newtheorem{defin}[thm]{Definition}

\newtheorem{lemma}[thm]{Lemma}
\newtheorem{conj}[thm]{Conjecture}
\newtheorem{cor}[thm]{Corollary}

\usepackage{hyperref}

\begin{document}

\title[Cycles of specified normalized length]{The number of cycles of specified normalized length in permutations}
 
\author{Michael Lugo}
\address{Department of Mathematics, University of Pennsylvania, 209 South 33rd Street, Philadelphia, PA 19104}
\email{mlugo@math.upenn.edu}

\subjclass[2000]{05A16, 60C05}

\date{September 15, 2009}

\maketitle

\begin{abstract} We compute the limiting distribution, as $n \to \infty$, of the number of cycles of length between $\gamma n$ and $\delta n$ in a permutation of $[n]$ chosen uniformly at random, for constants $\gamma, \delta$ such that $1/(k+1) \le \gamma < \delta \le 1/k$ for some integer $k$.  This distribution is supported on $\{0, 1, \ldots, k\}$ and has 0th, 1st, ..., $k$th moments equal to those of a Poisson distribution with parameter $\log {\delta \over \gamma}$.  For more general choices of $\gamma, \delta$ we show that such a limiting distribution exists, which can be given explicitly in terms of certain integrals over intersections of hypercubes with half-spaces; these integrals are analytically intractable but a recurrence specifying them can be given.   The results herein provide a basis of comparison for similar statistics on restricted classes of permutations. \end{abstract}

The distribution of the number of $k$-cycles in a permutation of $[n]$, for a fixed $k$, converges to a Poisson distribution with mean $1/k$ as $k \to \infty$.  In particular the mean number of $k$-cycles and the variance of the number of $k$-cycles are both $1/k$ whenever $n \ge k$ and $n \ge 2k$ respectively.  If instead of holding $k$ constant we let it vary with $n$, the number of $\alpha n$-cycles in permutations of $[n]$ approaches zero as $n \to \infty$ with $\alpha$ fixed.  So to investigate the number of cycles of long lengths, we must rescale and look at many cycle lengths at once.  In particular, we consider the number of cycles with length in some interval $[\gamma n, \delta n]$ as $n \to \infty$.  The expectation of the number of cycles with length in this interval is $\sum_{k = \gamma n}^{\delta n} 1/k$, which approaches the constant $\log \delta/\gamma$ as $n$ grows large.  By analogy with the fixed-$k$ case we might expect the number of cycles with length in this interval to be Poisson-distributed.  But this cannot be the case, because there is room for at most $1/\gamma$ cycles of length at least $\gamma n$, and the Poisson distribution can take arbitrarily large values.  In the case where $1/\gamma$ and $1/\delta$ lie in the same interval $[1/(k+1), 1/k]$ for some integer $k$, the limit distribution has the same first $k$ moments as Poisson($\log \delta/\gamma$).  For general $\gamma$ and $\delta$ the situation is considerably more complex but a limit distribution still exists.

In this paper, a ``random permutation of $[n]$'' will always mean a permutation chosen {\it uniformly} at random, and all expectations, distributions, etc. are relative to this choice of probability measure on $S_n$, the set of permutations of $n$.  The moments of the distributions of the number of $k$-cycles will be very useful, as we will initially express the limit distributions in terms of their moments.  We will need both individual and joint moments for the number of $k$-cycles, which we collect here.  Let $X_k^{(n)}$ denote the number of $k$-cycles in a random permutation of $[n]$.  Recall that $(z)_k = z(z-1)\ldots(z-k+1)$ denotes the ``$k$th falling power'' or ``$k$th factorial power'' of $z$.  This notation can be applied to random variables as well. 

\begin{prop}\label{prop:joint-moments}  Let $k_1, k_2, \ldots, k_s$ be distinct integers in $[1, n]$.  Let $r_1, \ldots, r_s$ be positive integers.  Then
\[
\E \left( \prod_{i=1}^s \left( X_{k_i}^{(n)} \right)_{r_i} \right) = \prod_{i=1}^s {1 \over k_i^{r_i}}
\]
if $n \ge \sum_{i=1}^s k_i r_i$, and zero otherwise. \end{prop}

\begin{proof} We construct the generating function, exponential in $z$ and ordinary in $u_1, \ldots, u_s$, which counts permutations by their size and number of $k_1, \ldots, k_s$-cycles.  This is
\[ P(z, u_1, \ldots, u_s) = {1 \over 1-z} \exp \left( \sum_{i=1}^s (u_i-1) {z^{k_i} \over k_i} \right). \]
The desired joint factorial moment is then
\[ { [z^n] \partial_{u_1}^{r_1} \cdots \left. \partial_{u_s}^{r_s} P(z, u_1, \ldots, u_s) \right|_{u_1 = \ldots = u_s = 1} \over [z^n] P(z, 1 \ldots, 1) } \]
and we note that each differentiation with respect to $u_i$ brings down a factor of $z^{k_i}/k_i$.  Thus we have
\[
\E \left( \prod_{i=1}^s (X_{k_i}^{(n)})_{r_i} \right) = {[z^n] { \prod_{i=1}^s \left( {z^{k_i}\over{k_i}} \right)^{r_i} \over 1-z} \over [z^n] {1 \over 1-z}} = [z^n] {z^{\sum_{i=1}^s k_i r_i} \over 1-z} \left( \prod_{i=1}^s k_i^{-r_i} \right). 
\]
where in the last equality we have used the fact that $[z^n](1-z)^{-1} = 1$ for all  $n \ge 0$.  The coefficient is $\prod_{i=1}^s k_i^{-r_i}$ if $n \ge \sum_{i=1}^s k_i r_i$ and $0$ otherwise, giving the desired result.
\end{proof}

In particular, we have the following corollary:
\begin{cor}\label{cor:moments} Let $X_k^{(n)}$ be the number of $k$-cycles in a permutation of $[n]$.  Then, choosing permutations uniformly at random, we have
 $\E ((X_k^{(n)})_r) = k^{-r}$ if $kr \le n$, and 0 otherwise.
\end{cor}
\begin{proof} This is the $s=1$ case of the previous proposition. \end{proof}

Note that Proposition \ref{prop:joint-moments} can be expressed in the following way, in light of Corollary \ref{cor:moments}: the joint factorial moments of numbers of $k$-cycles in random permutations are those of independent Poisson random variables with the same mean unless there is not enough room for the indicated cycles, in which case they are exactly zero.  Formulas for the joint power moments of the $X_k$ can be derived by expressing them as linear combinations of joint factorial moments.

Our major tool is the following theorem, which expresses the $r$th factorial moment of the number of cycles of a random permutation of $[n]$ with length in $[\gamma n, \delta n]$ as a certain $r$-fold integral.

\begin{thm}\label{thm:main} Fix $0 \le \gamma < \delta \le 1$.  Let $X^{(n)}$ be the number of cycles in a random permutation of $[n]$ having length in the interval $[\gamma n, \delta n]$.  Then
\[ \lim_{n \to \infty} \E(X^{(n)})_r = \int_{z_1 + \ldots + z_r \le 1 \atop z_i \in [\gamma, \delta]} {1 \over z_1 \cdots z_r} \: dz_1 \cdots dz_r. \]
 \end{thm}

\begin{proof} Let $X_k^{(n)}$ be the number of $k$-cycles of a random permutation of $[n]$.  Then $X^{(n)} = \sum_{k=\gamma n}^{\delta n} X_k^{(n)}$
and we can take the expectations of $r$th factorial moments to get
\[ \E \left( X^{(n)} \right)_r = \E \left( \left( \sum_{k=\gamma n}^{\delta n} X_k^{(n)} \right)_r \right). \]
This sum can be expanded using the multinomial theorem for falling powers.  We get
\[ \E \left( X^{(n)} \right)_r = \E \left( \sum_{l_{\gamma n} + \cdots + l_{\delta n} = r} (X_{\gamma n})_{l_{\gamma n}} \cdots (X_{\delta n})_{l_{\delta n}} {r \choose l_{\gamma n}, \cdots, l_{\delta n}} \right) \]
and we can bring the expectation inside the sum.  The termwise expectations are known from Proposition \ref{prop:joint-moments}, and so we have
\begin{equation}\label{eq:multinomial} \E(X^{(n)})_r = \sum_{l_{\gamma n} + \cdots + l_{\delta n} = r \atop \sum_{k = \gamma n}^{\delta n} k l_k \le n} \left[ {r \choose l_{\gamma n}, \cdots, l_{\delta n}} \prod_{k = \gamma n}^{\delta n} \left( {1 \over k} \right)^{l_k}   \right] \end{equation}
Now, we consider the multinomial expansion 
\begin{equation}\label{eq:rth-power} \left( \sum_{k=\gamma n}^{\delta n} {1 \over k} \right)^r = \sum_{l_{\gamma n} + \cdots + l_{\delta n} = r} \left[ {r \choose l_{\gamma n}, \cdots, l_{\delta n}} \prod_{k = \gamma n}^{\delta n} \left( {1 \over k} \right)^{l_k}   \right] \end{equation}
The expansion has a term $1/(k_1 \ldots k_r)$ for each $r$-tuple $(k_1, \ldots, k_r)$ in $[\gamma n, \delta n]^r$.   This can be interpreted as a Riemann sum for the $r$-fold integral 
\[ \int_{\gamma n}^{\delta n} \cdots \int_{\gamma n}^{\delta n} {1 \over w_1 \cdots w_r} \: dw_1 \cdots dw_r \]
The restriction $\sum_k kl_k \le n$ cuts off that part of the region of summation where $w_1 + \cdots + w_r > n$. Thus the actual sum (\ref{eq:multinomial}) is a Riemann sum for 
\[ \int \ldots \int {1 \over w_1 \cdots w_r} \: dw_1 \cdots dw_r \]
where the $r$-fold integral is over $w_1 + \ldots + w_n \in [\gamma n, \delta n], w_1 + \ldots + w_r \le n$.  The change of variables $z_i = w_i/n$ gives the desired result. 
\end{proof}

\begin{prop}\label{prop:longest-cycles} Fix $\alpha > 1/2$.  As $n \to \infty$, the probability that a randomly chosen permutation of $[n]$ has a cycle of length at least $\alpha n$ approaches $-\log \alpha$. \end{prop}
\begin{proof} We apply Theorem \ref{thm:main} to get
\[ \lim_{n \to \infty} \E(X^{(n)}) = \int_\alpha^1 {1 \over z} \: dz = - \log \alpha. \]
A permutation of $[n]$ can have at most one cycle of length longer than $n/2$, so the probability of having such a cycle is equal to the expected number of them.
\end{proof}

We can compare this to a number-theoretic result: the expected number of prime factors of a random integer in $[1,N]$ which are at least $N^\alpha$, for $\alpha > 1/2$, also approaches $-\log \alpha$ as $N \to \infty$. This is but one example of an analogy between prime factorizations of random integers and cycle structure of permutations, developed by Granville in \cite{G-anatomy}. 
This is the simplest example of our general method.  We know that the distribution of $X$ is concentrated on two values; thus knowing $\E (X^{(n)})_0$ and $\E (X^{(n)})_1$ suffices to give the distribution of $X$.  In general, if we know that $X$ is concentrated on $k$ values, finding $\E (X^{(n)})_0, \E (X^{(n)})_1, \ldots, \E (X^{(n)})_{k-1}$ gives a system of $k$ linear equations in $k$ unknowns which can be solved to determine the distribution of $X$.  In order to make stating results easier, we make the following definition.

\begin{defin} We say a random variable $X$ has quasi-Poisson($r,\lambda$) distribution if $\E((X)_k) = \lambda^k$ for $k = 0, 1, \ldots, r$ and $X$ is supported on $\{ 0, 1, \ldots, r \}$. \end{defin}

The $k$th factorial moment of a Poisson($\lambda$) random variable is $\lambda^k$.  So in a sense, the quasi-Poisson random variables are trying to be Poisson, subject to an upper limit on their value.  Let $\pi_i(r, \lambda)$ be the probability that a quasi-Poisson($r,\lambda$) random variable has value $i$.  Our knowledge of the moments allows us to set up a system of equations to find $\pi_i(r, \lambda)$.  The solution is given in the following theorem.
\begin{thm}\label{thm:quasi-poisson-values} The probability that a quasi-Poisson$(r,\lambda)$ random variable has value $i$ is
\begin{equation}\label{eq:quasi-poisson-values} \pi_i(r, \lambda) = \sum_{j=i}^r {j \choose i} {1 \over j!} (-1)^{j-i} \lambda^j.\end{equation}
 \end{thm}
 We begin by recalling the following lemma.   

\begin{lemma}\label{lem:binomial-matrix} Let $M = M_n, N = N_n$ be $(n+1)$ by $(n+1)$ matrices such that $M_{ij} = {j \choose i}, N_{ij} = {j \choose i} (-1)^{j+i}$, where the rows and columns of $M$ and $N$ are indexed by $0, 1, \ldots, n$.  Then $MN = I$, the identity matrix. \end{lemma}

For a proof, see \cite[p. 66-67]{S}.

\begin{proof}[Proof of Theorem \ref{thm:quasi-poisson-values}] The factorial moments specified in the definition of quasi-Poisson random variables give
\begin{equation}\label{eq:quasi-poisson-system}(1, \lambda, \ldots, \lambda^r)^T = A_r (\pi_0(r,\lambda), \pi_1(r,\lambda), \ldots, \pi_r(r,\lambda))^T \end{equation}
where $A_r$ is an $(r+1)$ by $(r+1)$ matrix, with rows and columns indexed by $0, 1, \ldots, r$, and $(A_r)_{ij} = (j)_i$.  The $k$th entry when the right-hand side of (\ref{eq:quasi-poisson-system}) is $\sum_{k=0}^r (k)_i \pi_k(r, \lambda)$, which is the expectation of $(X)_i$ when $X$ is quasi-Poisson.  This matrix is obtained from the $M_r$ of Lemma \ref{lem:binomial-matrix} by multiplying all the entries in column $i$ by $i!$.  By Lemma \ref{lem:binomial-matrix} its inverse is obtained from $N_r$ by dividing all the entries in row $j$ by $j!$.  Thus, we have
\[
B_r(1, \lambda, \ldots, \lambda^r)^T = (\pi_0(r,\lambda), \pi_1(r,\lambda), \ldots, \pi_r(r,\lambda))^T
\]
where $B_r = N_r^{-1}$.  Thus $(B_r)_{ij} = {j \choose i} {1 \over j!} (-1)^{j+i}$ and this is the desired result in matrix form. \end{proof}

The sum (\ref{eq:quasi-poisson-values}) giving $\pi_i(r,\lambda)$ consists of the first $r-i$ nonzero terms of the Maclaurin series for $(z^i/i!)e^{-z}$, evaluated at $z=\lambda$.  Thus if $r$ is large, then $\pi_i(r,\lambda)$ approximates the corresponding probability for Poisson random variables.

While the theorems given here only invoke quasi-Poisson distributions of mean at most $\log 2$, in fact we have
\begin{thm} Quasi-Poisson($r,\lambda$) random variables exist for every positive integer $r$ and real number $\lambda \in [0,1]$, and no other choices of $\lambda$. \end{thm}
\begin{proof} The system of equations (\ref{eq:quasi-poisson-system}) which gives the quasi-Poisson distribution is solved in Theorem \ref{thm:quasi-poisson-values}; we need to show that
\begin{equation}\label{eq:system-restated}
\pi_i(r,\lambda) = \sum_{j=i}^r {j \choose i} {1 \over j!} (-1)^{j+i} \lambda^j
\end{equation}
is nonnegative for all $i$ exactly when $\lambda \in [0,1]$.  We have
\begin{eqnarray*}
\pi_{r-1}(r,\lambda) &=& {r-1 \choose r-1} {1 \over (r-1)!} (-1)^{2r-2} \lambda^{r-1} + {r \choose r-1} {1 \over r!} (-1)^{2r-1} \lambda^r \\
&=& {1 \over (r-1)!} \lambda^{r-1} (1-\lambda)
\end{eqnarray*}
and if $\lambda>1$ this is negative.  So the solution previously given for this system does  not give the distribution of a random variable; as noted in Lemma \ref{lem:binomial-matrix} the system is nonsingular, so this is the only solution to the system.  If $0 < \lambda < 1$, we note that the terms of (\ref{eq:system-restated}) are alternately negative and positive, with the first term positive.  Thus to show $\pi_i(r, \lambda) > 0$, it suffices to show that the terms are decreasing in absolute value as $j$ increases.  That is, we need
\[
{j \choose i} {1 \over j!} \lambda^j > {j+1 \choose i} {1 \over (j+1)!} \lambda^{j+1}
\]
and the left-hand side of this equation, divided by its right-hand side, is $(j+1-i)/\lambda$; since $j \ge i$ and $\lambda < 1$ this is greater than $1$, giving the desired inequality. \end{proof}

The quasi-Poisson($r,1$) distribution is well-known under another name in the study of permutations.  It is the distribution of the number of fixed points of a permutation of $[r]$.

We will generally prove convergence of a sequence of random variables to a quasi-Poisson by proving that the factorial moments of that sequence are converging to the factorial moments of a quasi-Poisson.   It is known that the moments of a distribution with finite support uniquely determine the distribution \cite[p. 778]{FS}.  It is also the case that if $F_n(x)$ for $n = 0, 1, 2, \ldots$ are the distribution functions of random variables and 
\[ \lim_{n \to \infty} \int_{-\infty}^\infty (x)_k dF_n(x) = \int_{-\infty}^\infty (x)_k dF(x) \]
and $F$ is characterized by its moments, then the $F_n$ converge in distribution to $F$ \cite[Thm. 30.2]{Bi}.  Therefore to show that a sequence of random variables converge to a quasi-Poisson, it suffices to show that the moments converge to those of the quasi-Poisson.  The proof of the following theorem is an example. 

\begin{thm}\label{thm:box-integral} Fix $\gamma, \delta$ such that ${1 \over k+1} \le \gamma < \delta \le {1 \over k}$ for some integer $k$.  (Alternatively, $\lfloor \delta^{-1} \rfloor + 1 = \lceil \gamma^{-1} \rceil$.)  Let $X^{(n)}$ be a random variable on $S_n$ with uniform measure, with $X^{(n)}(\pi)$ equal to the number of cycles of the permutation $\pi$ with length in $[\gamma n, \delta n]$.  Then as $n \to \infty$, $X_n$ converges in distribution to the quasi-Poisson$(k,\log \delta/\gamma)$ distribution.
\end{thm}

\begin{proof} It suffices to show that the 0th through $k$th factorial moments of $X^{(n)}$ approach those of the quasi-Poisson, i. e.
 \[ \lim_{n \to \infty} \E((X^{(n)})_r) = (\log \delta/\gamma)^r. \]
We apply Theorem \ref{thm:main}; the desired limit is
\[ \int_{z_1 + \ldots + z_r \le 1 \atop z_i \in [\gamma, \delta]} {1 \over z_1 \cdots z_r} \: dz_1 \cdots dz_r \]
and this integral is actually over an $r$-dimensional box $[\gamma, \delta]^r$, since the condition $z_1 + \cdots + z_r \le 1$ is always satisfied.  The integral factors into
\[ \left( \int_\gamma^\delta {1 \over z} \: dz \right)^r = \log \left( {\delta \over \gamma} \right)^r \]
and these are the factorial moments of the Poisson (or quasi-Poisson), proving the theorem.
\end{proof}

For example, in a random permutation of $[n]$, for $n$ large, how many cycles have length between $n/4$ and $n/3$?  We know that the limiting distribution is quasi-Poisson($3,\lambda$) with $\lambda = \log 4/3$; the values can be found explicitly, and are
\begin{eqnarray*}
\pi_0 &=& 1 - \lambda + \lambda^2/2 - \lambda^3/6 = 0.7497\ldots \\
\pi_1 &=& \lambda - \lambda^2 + \lambda^3/2 = 0.2168\ldots \\
\pi_2 &=& \lambda^2/2 - \lambda^3/2 = 0.0295\ldots \\
\pi_3 &=& \lambda^3/6 = 0.0040\ldots
\end{eqnarray*}

One shortcoming of Theorem \ref{thm:box-integral} (and, implicitly, Theorem \ref{prop:longest-cycles}), which the reader may have noted, is that we require $\gamma$ and $\delta$ to be in the same interval of the form $[{1 \over k+1}, {1 \over k}]$ for some integer $k$.  This is not accidental; the expressions for the limiting probabilities become much more complicated if this is not the case.  However, such expressions still exist.

\begin{prop}  Fix an integer $i$.  The probability that a permutation has $i$ cycles of length in $[ \gamma n, \delta n ]$ for any $0 < \gamma < \delta \le 1$, approaches a limit as $n \to \infty$. \end{prop}

\begin{proof} We apply Theorem \ref{thm:main} to see that in a random permutation of $[n]$, the number of cycles with length in $[\gamma n, \delta n]$ has an $r$th factorial moment which approaches some finite limit as $n \to \infty$.  In particular the $0$th through $\lfloor \gamma^{-1} \rfloor$th moments of the limiting distribution can be found.  These give a system of $\lfloor \gamma^{-1} \rfloor + 1$ equations in the same number of unknowns, $\P(X=0), \ldots, \P(X=\lfloor \gamma^{-1} \rfloor)$, which can be solved to determine the limiting probabilities.  \end{proof}

The integrals of Theorem \ref{thm:main}, when ${1 \over k+1} \le \gamma < \delta \le {1 \over k}$, are integrals over $r$-cubes and thus factor easily.  In more general circumstances, Theorem \ref{thm:main} gives an integral over some sliced cube, that is, that part of $[\gamma, \delta]^r$ in which the sum of the coordinates is less than some constant.  Under these circumstances such a factorization is not possible. 

Let $1/3 \le \gamma \le 1/2 \le \delta \le 1$.  By Theorem \ref{thm:main}, the probability that a permutation has $i$ cycles of length between $\gamma n$ and $\delta n$, for $i = 0, 1, 2$, approaches a limit as $n \to \infty$.   We let these limits be denoted by $p_i(\gamma, \delta)$ and obtain explicit expressions for $p_0, p_1, p_2$. 

Let $X^{(n)}$ be the number of cycles of a random permutation of $[n]$ with length in $[\gamma n, \delta n]$.  Let $q_i(\gamma, \delta)$ denote $\lim_{n \to \infty} \E(X^{(n)})_i$ for $i = 0, 1, 2$.  We note that
\begin{equation} q_0 = p_0+p_1+p_2, q_1 = p_1 + 2p_2, q_2 = 2p_2 \end{equation}
which can be solved for the $p_i$ to give 
\begin{equation}\label{eq:q-to-p} p_0 = q_0 - q_1 + q_2/2, p_1 = q_1 - q_2, p_2 = q_2/2. \end{equation}   

 Clearly $q_0(\gamma, \delta) = 1$ for all $\gamma, \delta$.  From Theorem \ref{thm:main} we have
\[ q_1 = \int_{\gamma}^{\delta} {1 \over z} \: dz = \log {\delta \over \gamma}. \]

Finally, we have
\[ q_2 = \int \int_{x,y \in [\gamma, \delta] \atop x+y \le 1} {1 \over xy} \: dx \: dy. \]
We must separate into two cases based on the relationship of $\gamma + \delta$ to $1$.  If $\gamma + \delta \ge 1$, then the region of integration is a triangle.  We have the iterated integral
\begin{equation}\label{eq:iterated-integral} \int_r^{1-r} \int_r^{1-x} {1 \over xy} \: dy \: dx = - \log r \log(1-r) - Li_2(r) + Li_2(1-r) + (\log r)^2 \end{equation}
where $Li_2$ is the dilogarithm,
\begin{equation}\label{eq:li2-def} Li_2(z) = \sum_{k=1}^\infty {z^k \over k^2} = \int_z^0 {\log(1-t) \over t} \: dt. \end{equation}
  If $\gamma + \delta \ge 1$, then we can just substitute $\gamma$ for $r$ in this integral to get $q_2$.  (Note that $q_2$ does not depend on $\delta$ in this case.)  If $\gamma + \delta < 1$, then we can break the region of integration into the three rectangles $R_1 = [\gamma, 1-\delta]^2, R_2 = [\gamma, 1-\delta] \times [1-\delta, \delta], R_3 = [1-\delta, \delta] \times [\gamma, 1-\delta]$ and the triangle $T = \{ x,y > 1-\delta, x+y < 1 \}$, as illustrated in Figure \ref{fig:integral-division}.

\begin{figure}
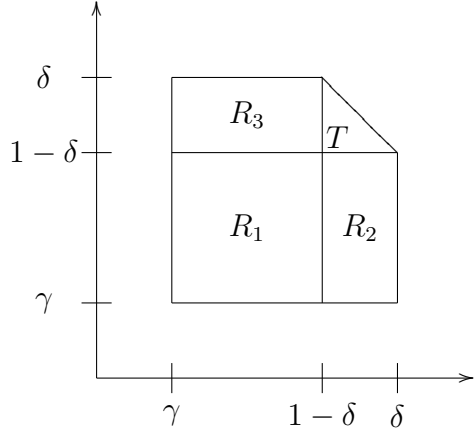

\[
\xy
{\ar (0,0)*{}; (50,0)*{}};
{\ar (0,0)*{}; (0,50)*{}};
(10,10)*{}; (40,10)*{} **\dir{-};
(10,30)*{}; (40,30)*{} **\dir{-};
(10,40)*{}; (30,40)*{} **\dir{-};
(10,10)*{}; (10,40)*{} **\dir{-};
(30,10)*{}; (30,40)*{} **\dir{-};
(40,10)*{}; (40,30)*{} **\dir{-};
(30,40)*{}; (40,30)*{} **\dir{-};
(10,-2)*{}; (10,2)*{} **\dir{-};
(30,-2)*{}; (30,2)*{} **\dir{-};
(40,-2)*{}; (40,2)*{} **\dir{-};
(-2,10)*{}; (2,10)*{} **\dir{-};
(-2,30)*{}; (2,30)*{} **\dir{-};
(-2,40)*{}; (2,40)*{} **\dir{-};
(20,20)*{R_1};
(35,20)*{R_2};
(20,35)*{R_3};
(32,32)*{T};
(10,-5)*{\gamma};
(30,-5)*{1-\delta};
(40,-5)*{\delta};
(-7,10)*{\gamma};
(-7,30)*{1-\delta};
(-7,40)*{\delta};
\endxy
\]
\caption{Division of the region of integration into rectangles and a triangle.}\label{fig:integral-division}
\end{figure}

The integrals over rectangles are straightforward; we have already considered the integral over a triangle in (\ref{eq:iterated-integral}). Putting everything together, we get
\[ q_2(\gamma, \delta) = \log {1-\delta \over \gamma} \log {\delta^2 \over (1-\delta) \gamma} - \log \delta \log(1-\delta) - Li_2(1-\delta) + Li_2(\delta) + (\log(1-\delta))^2 \]
if $\gamma + \delta < 1$, and
\[ q_2(\gamma, \delta) = - \log \gamma \log (1-\gamma) - Li_2(\gamma) + Li_2(1-\gamma) + (\log \gamma)^2 \]
if $\gamma + \delta \ge 1$.  

Finally, from these formulas for the $q_i$ we can obtain formulas for the $p_i$ using (\ref{eq:q-to-p}).

We now specialize to the case $\delta = 1$.  Fix some notation: let $Q_k(\gamma) = q_k(\gamma, 1)$ and let $P_k(\gamma) = p_k(\gamma, 1)$.  Then we get the formulas
\begin{equation}\label{eq:Q-formulas} Q_0(\gamma) = 1, Q_1(\gamma) = -\log \gamma, Q_2(\gamma) = -\log \gamma \log(1-\gamma) - Li_2(\gamma) + Li_2(1-\gamma) + (\log \gamma)^2 \end{equation}
from which we can derive formulas for the $P_i(\gamma)$.  These formulas give the limiting probability of having $0, 1$ or $2$ cycles longer than length $\gamma n$ in a permutation of length $n$, for $\gamma \in [1/3, 1/2]$; these probabilities are analytic functions of $\gamma$ in that interval.  

In particular, we consider
\[ P_1(\gamma) = -\log(\gamma) + \log(\gamma) \log(1-\gamma) - \log(\gamma)^2 + Li_2(\gamma) - Li_2(1-\gamma) \]
which applies over $\gamma = [1/3, 1/2]$.  We have $P_1(\gamma) = -\log(\gamma)$ for $1/2 \le \gamma \le 1$; the probabilities $P_i(\gamma)$ for $i = 0, 1, 2$ are shown in Figure \ref{fig:P-plot}.  Differentiating with respect to $\gamma$ gives
\[ {d \over d \gamma} P_1(\gamma) = {-1 \over \gamma} + {2 \log (1-\gamma) \over \gamma} - {2 \log \gamma \over \gamma} \]
where we have used the fact ${d \over dz} Li_2(z) = \log(z)/(1-z)$, which follows from the integral definition (\ref{eq:li2-def}).  Solving for $P_1^\prime(\gamma)=0$ gives $\gamma_0 = (1+e^{1/2})^{-1} = 0.3775\ldots$.  This is the value of $\gamma$ that maximizes the probability of having {\it exactly} one cycle of normalized length longer than $\gamma$.  We note that it is close to the value of $e^{-1} = 0.3678\ldots$ that might be naively expected, since the expected number of cycles longer than $n/e$ is $-\log 1/e = 1$.    We have $P_1(\gamma_0) = 0.8285\ldots, P_0(\gamma_0) = 0.0987\ldots, P_2(\gamma_0) = 0.0728\ldots$; thus {\it most} permutations of $[n]$ have {\it exactly one} cycle longer than $\gamma_0 n$.  One might expect the limiting distribution of the number of cycles longer than $n/e$ (or $\gamma_0 n$) to be Poisson, or at least quasi-Poisson, but the distribution of the number of long cycles is much more strongly peaked.  This is because not only is it impossible to have three or more such long cycles, but it is difficult to fit even two; thus to achieve a mean near 1, the value 1 must actually occur quite often. 

\begin{figure}\includegraphics[width=0.5\textwidth]{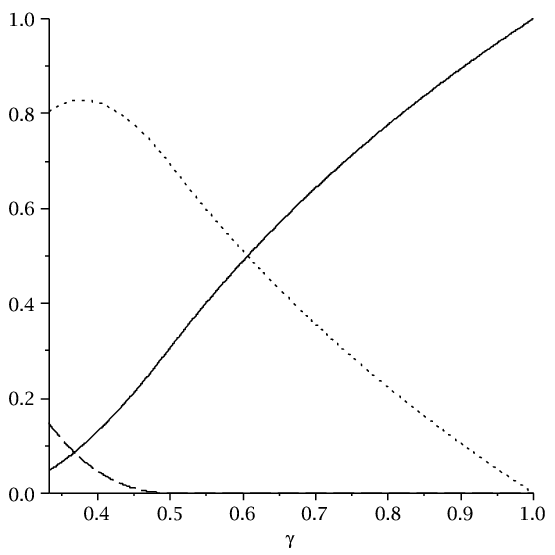} \caption{$P_0(\gamma)$ (solid), $P_1(\gamma)$ (dotted), $P_2(\gamma)$ (dashed) for $1/3 \le \gamma \le 1$.  }\label{fig:P-plot} \end{figure}

Finally, we consider some analytic properties of the functions $Q_k$.
\begin{prop}\label{prop:small-simplex}  We have
\[ \lim_{\gamma \to (1/k)^-} {Q_k(\gamma) \over (1-k\gamma)^k} = {k^k \over k!}. \]
\end{prop}
\begin{proof} Observe that
 \[ Q_k(\gamma) = \int \cdots \int {1 \over z_1 \cdots z_k} \: dz_1 \cdots dz_k \]
where the integral is over $z_1, \ldots, z_k > \gamma, z_1 + \cdots + z_k < 1$.  The region of integration is a right $k$-simplex with vertex at $(\gamma, \gamma, \ldots, \gamma)$; each side parallel to a coordinate axis has length $1-k\gamma$, so its volume is $(1-k\gamma)^k/k!$.  The integrand, in this region, is bounded below by $\gamma^{-k}$, its value at the vertex of the simplex.  It is bounded above by $k^k$, its value at the center of the opposing face.  Therefore
\[ {\gamma^{-k} \over k!} \le {Q_k(\gamma) \over (1-k\gamma)^k} \le {k^k \over k!} \]
and taking limits as $\gamma \to (1/k)^-$ gives the desired result.
\end{proof}

\begin{prop}\label{prop:Q-recurrence} The functions $Q_k$ satisfy the recurrence
 \begin{equation}\label{eq:Q-recurrence} Q_{k+1}(\gamma) = \int_\gamma^{1-k\gamma} {1 \over z} Q_k \left( {\gamma \over 1-z} \right) \: dz \end{equation}
for $\gamma < 1/(k+1)$, and $Q_{k+1}(\gamma) = 0$ otherwise, with the initial condition $Q_0(\gamma) = 1$.
\end{prop}
Note that this recovers the previous formulas (\ref{eq:Q-formulas}) for $Q_1$ and $Q_2$.  
\begin{proof} We have $Q_0(\gamma) = 1$ since the zeroth falling power is identically 1.  To derive the recurrence, we recall the result of Theorem \ref{thm:main}.  This gives
\[ Q_{k+1}(\gamma) = \int {1 \over z_1 \cdots z_{k+1}} dz_1 \cdots dz_{k+1} \]
where the integral is over the region $z_1 + \cdots + z_{k+1} \le 1, z_i \in [\gamma, 1]$.  The region of integration is a sliced $(k+1)$-cube, i. e. that part of a $(k+1)$-cube lying below a plane $z_1 + \ldots + z_{k+1} = c$; all its $k$-dimensional cross-sections are themselves sliced $k$-cubes.  We pull out $z_{k+1}$ to get a single integral of a $k$-fold integral,
\begin{equation}\label{eq:iterated-integral-1}  Q_{k+1}(\gamma) = \int_\gamma^{1-k\gamma} {1 \over z_{k+1}} \left[ \int {1 \over z_1 \cdots z_k} \: dz_1 \ldots dz_k \right] dz_{k+1} \end{equation}
where the inner integral is over the sliced cube $z_1 + \ldots + z_k \le 1-z_{k+1}, z_i \in [\gamma, 1]$.  To explain the upper bound on the outer integral, note that the region of integration in the inner integral is empty if $z_{k+1} \ge 1-k\gamma$.  

In (\ref{eq:iterated-integral-1}), make the change of variables $w_j = z_j/(1-z_{k+1})$ for $j = 1, \ldots, k$.  This gives 
\begin{equation}\label{eq:iterated-integral-2} Q_{k+1}(\gamma) = \int_\gamma^{1-k\gamma} {1 \over z_{k+1}} \left[ \int {1 \over w_1 \cdots w_k} \: dw_1 \cdots dw_k \right] \: dz_{k+1} \end{equation}
where the inner integral is over the simplex $w_1 + \cdots + w_k \le 1, w_i \in [\gamma/(1-z_{k+1}), 1/(1-z_{k+1})]$.  But in fact none of the $w_i$ can exceed 1, since they are all positive and their sum is at most 1.  Thus the inner integral in (\ref{eq:iterated-integral-2}) is exactly $Q_k(\gamma/(1-z))$, which yields (\ref{eq:Q-recurrence}). \end{proof}

\begin{cor} For each $k \ge 1$, $Q_k(\gamma)$ is a $C^\infty$ function on $(0,1)$, except that it is $C^{k-1}$ but not $C^k$ at $\gamma = 1/k$. \end{cor}
\begin{proof} We proceed by induction.  Note that for $\gamma \in (0,1)$, we have $Q_1(\gamma) = \int_\gamma^1 {1 \over z} \: dz = -\log \gamma$, and $Q_1(\gamma) = 0$ for $\gamma \ge 1$.   Thus $Q_1(\gamma)$ is $C^\infty$ except at 1, where it is $C^0$. 

Now, assume $Q_k(\gamma)$ is $C^\infty$ on $(0,1)$, except that it is $C^{k-1}$ but not $C^k$ at $\gamma = 1/k$.  Then we have
\[ Q_{k+1}(\gamma) = \int_\gamma^{1-k\gamma} {1 \over z} Q_k \left( {\gamma \over 1-z} \right) \: dz \]
for $\gamma \in (0, 1/(k+1))$.  This is the integral of a $C^\infty$ function between limits that are $C^\infty$ in $\gamma$; thus it is $C^\infty$.  For $\gamma > 1/(k+1)$, we have $Q_{k+1}(\gamma) = 0$ from Proposition \ref{prop:Q-recurrence} so $Q_{k+1}(\gamma) = 0$ on $(1/(k+1), 1)$ and the function is $C^\infty$ there.  Finally, we observe from Proposition \ref{prop:small-simplex} that $Q_{k+1}(\gamma)$ is $C^k$ but not $C^{k+1}$ at $1/(k+1)$ -- the $k$th derivatives and all lower derivatives on either side of $1/(k+1)$ are both zero, but the $(k+1)$st derivatives differ. 
\end{proof}

\begin{cor} $P_j(\gamma)$ is $C^{k-1}$ but not $C^k$ at $1/k$ for all $k > j$, and $C^\infty$ elsewhere. \end{cor}
\begin{proof} We have 
 \[ P_k(\gamma) = {1 \over k!} \sum_{j=0}^\infty {(-1)^j \over j!} Q_{k+j}(\gamma) \]
and so the non-$C^\infty$ points of $P_k$ are exactly those of $Q_k, Q_{k+1}, \ldots$.
\end{proof}

It would be interesting to derive, from the recurrence formula in Proposition \ref{prop:Q-recurrence} or otherwise, more numerical results about the $Q_i$ or $P_i$ -- for example, for which $\gamma$ is $P_2(\gamma)$ maximized?  (Related number-theoretic functions, such as the Buchstab and Dickman functions, can be computed, but clever numerical tricks are necessary; see \cite[Ch. 5]{W} and the references therein.)

Granville, inspired by results in number theory, has shown \cite[Theorem 5]{G} that the proportion of permutations of $[n]$ with all cycles having length at least $\alpha n$ is given by
\[
{\omega(\alpha^{-1}) \over \alpha n} + O \left( {\log \log n \over n^2} \right)
\]
where $\omega$ is the Buchstab function, given by
\[
\omega(u) = 1/u \text{ for $1 \le u \le 2$}, \omega(u) = {1 \over u} \int_1^{u-1} \omega(t) \: dt \text{ for $u \ge 2$}.
\]
Proposition \ref{prop:Q-recurrence} is reminiscent of this result, and of similar results on the prime factorizations of integers.  However, classical number-theoretic results in this vein have focused on the numbers of integers near $n$ with all factors in some fixed normalized interval $[n^\gamma, n^\delta]$.  The case $\gamma = 0$ (that is, integers with all factors less than $n^\delta$) was considered by Dickman \cite{D}, and that of $\delta = 1$ (all factors greater than $n^\gamma$) by Buchstab \cite{B}; the general case was treated by Friedlander \cite{F}.  Wolczuk's thesis \cite{W} compiles many results on the Buchstab function.  There do not seem to be results considering the probability that an integer near $n$ has a specified {\it number} of prime factors in $[n^\gamma, n^\delta]$, which would be the number-theoretic analogue of the results given here.  Similarly, results on the sizes of components of combinatorial structures have in general focused on the lengths of the longest or shortest components. Shepp and Lloyd looked at the longest cycles of permutations \cite{SL}; more recent work of Panario and Richmond \cite{PR1, PR2} has extended this to smallest and largest components of more general decomposable structures.

The distributions explored here were first encountered during the writing of \cite{L}.  That paper considers the cycle structure of permutations chosen uniformly at random from those with all cycle lengths odd, or all even, or from the Ewens distribution with parameter $\sigma$.  (The Ewens distribution \cite{E} assigns weight $\sigma^{c(\pi)}$ to a permutation $\pi$, where $c(\pi)$ is the number of cycles of $\pi$, and chooses each permutation with probability proportional to its weight.  In the case of the Ewens distribution, the following conjecture seems reasonable:
\begin{conj} The expected number of cycles of length in $[\gamma n, \delta n]$ of a permutation of $[n]$ chosen from the Ewens distribution approaches
 \[ \lambda = \int_\gamma^\delta {1 \over x} (1-x)^{\sigma-1} \: dx \]
as $n \to \infty$.  Furthermore, in the case where $1/(k+1) \le \gamma < \delta < 1/k$ for some positive integer $k$, the distribution of the number of cycles converges in distribution to quasi-Poisson$(k,\lambda)$. 
\end{conj}
Furthermore, in \cite[Cor. 2.10]{L} it is shown that the proportion of elements of a random permutation of $n$ selected according to the Ewens distribution which are in cycles of length in $[\gamma n, \delta n]$ approaches $(1-\gamma)^\sigma - (1-\delta)^\sigma$ as $n$ gets large.  It is also known that the same is true for permutations of $n$ selected uniformly from all those with all cycle lengths even, or from all those with all cycle lengths odd \cite[Thm. 3.5]{L}.  It seems reasonable to conjecture that this correspondence should hold at least so far as to give that these classes of permutations satisfy the previous conjecture with $\sigma = 1/2$.   Similar distributions also may be obtained for other combinatorial structures in which components have size comparable with the size of the entire structure, including the so-called logarithmic combinatorial structures \cite{ABT}. 

{\it Acknowledgments.} Mirko Visontai pointed out that Lemma \ref{lem:binomial-matrix} is well-known and provided the reference to Stanley's text.  Correspondence with Warren Ewens motivated the conjecture in the conclusion.


\begin{thebibliography}{1}
\bibitem[ABT03]{ABT} Richard Arratia, A. D. Barbour, and Simon Tavar\'e. {\it Logarithmic combinatorial structures: a probabilistic approach.} European Mathematical Society, 2003.
\bibitem[Buc49]{B} A. A. Buchstab.  On those numbers in an arithmetic progression all prime factors of which are small in order of magnitude.  {\it Doklady Akad. Nauk. SSSR} 67 (1949) 5-8.
\bibitem[Bil95]{Bi} Patrick Billingsley.  {\it Probability and measure}, 3rd edition.  Wiley, 1995.
\bibitem[Dic30]{D} K. Dickman, On the frequency of numbers containing prime factors of a certain relative magnitude, {\it Arkiv f\"{o}r Matematik, Astronomi och Fysik} 22A:10 (1930) 1-14.
\bibitem[Ew72]{E}  Warren Ewens. The sampling theory of selectively neutral alleles. {\it Theoret. Population Biol.} 3 (1972), 87-112.
\bibitem[Fri76]{F} John B. Friedlander.  Integers free from large and small primes.  {\it Proc. London Math Soc.} 3 (1976) 565-576.
\bibitem[FS09]{FS} Philippe Flajolet and Robert Sedgewick.  {\it Analytic combinatorics.}  Cambridge University Press, 2009.
\bibitem[Gra06]{G} Andrew Granville.  Cycle lengths in a permutation are typically Poisson.  {\it Electronic Journal of Combinatorics} 13 (2006), R107.
\bibitem[Gra09+]{G-anatomy} Andrew Granville.  Anatomy of integers and permutations.  Preprint.   Available online at http://www.dms.umontreal.ca/~andrew/preprints.html .
\bibitem[Lug09]{L} Michael Lugo. Profiles of permutations.  {\it Electronic Journal of Combinatorics} 16 (2009), R99.
\bibitem[PR01a]{PR1} Daniel Panario and Bruce Richmond. Exact largest and smallest size of components in decomposable structures. {\it Algorithmica} 31 (2001) 413-432.
\bibitem[PR01b]{PR2} Daniel Panario and Bruce Richmond. Smallest components in decomposable structures: exp-log class.  {\it Algorithmica} 29 (2001) 205-226.
\bibitem[SL66]{SL} L. A. Shepp and S. P. Lloyd. Ordered cycle lengths in a random permutation. {\it Trans. Amer. Math. Soc.} 121 (1966) 340-357.
\bibitem[Sta99]{S} Richard P. Stanley. {\it Enumerative combinatorics}, volume 1. Cambridge University Press, 1999. 
\bibitem[Wol03]{W} Dan Wolczuk, Intervals with few prime numbers.  Master's thesis, University of Waterloo, 2003.   Available online at etd.uwaterloo.ca/etd/dstwolcz2004.ps .
\end{thebibliography}
\end{document}